\documentclass[12pt,a4paper,oneside]{amsart}
\usepackage[left=1.1in, right=0.8in, top=0.9in, bottom=0.81in]{geometry}
\usepackage{graphicx}
\usepackage{setspace}
\usepackage{indentfirst} 
\usepackage{cases}
\usepackage{wrapfig}
\usepackage[toc,page]{appendix}
\usepackage{amsmath}
\usepackage{amsthm}
\usepackage{amssymb}
\usepackage{enumerate}
\usepackage{mathrsfs}
\usepackage{dcolumn}
\usepackage{mathtools}

\newcolumntype{L}{D{.}{.}{2,5}}
\theoremstyle{plain}
\newtheorem{thm}{Theorem}

\newtheorem{proposition}{Proposition}
\newtheorem{lemma}{Lemma}
\newtheorem{corollary}[proposition]{Corollary}

\newtheorem{example}{Example}

\DeclareMathOperator{\cl}{cl}

\DeclareMathOperator{\co}{co}

\DeclareMathOperator{\CAT}{CAT}

\theoremstyle{definition}

\DeclareMathOperator{\Limsup}{\text{Ls}}
\DeclareMathOperator{\Liminf}{\text{Li}}

\DeclareMathOperator{\Lim}{\text{Lim}}
\DeclareMathOperator{\thr}{thr}
\DeclareMathOperator{\inter}{int}

\usepackage{cite}
 \usepackage{hyperref}
 \hypersetup{
    colorlinks=true,
    linkcolor=black,
    filecolor=magenta,      
    urlcolor=cyan}

\begin{document}

\title[Compact sets and the closure of their convex hulls in CAT(0) spaces]{Compact sets and the closure of their convex hulls in CAT(0) spaces}
%
%
\author{
{Arian B\"erd\"ellima}}
 \address{Technische Universit\"at Berlin}\address{10623 Berlin, Germany}

 \begin{abstract}
 We study the closure of the convex hull of a compact set in a complete $\CAT(0)$ space. First we give characterization results in terms of compact sets and the closure of their convex hulls for locally compact $\CAT(0)$ spaces that are either regular or satisfy the geodesic extension property. Later inspired by a geometric interpretation of Carath\'eodory's Theorem we introduce the operation of threading for a given set. We show that threading exhibits certain monotonicity properties with respect to intersection and union of sets. Moreover threading preserves compactness. Next from the commutativity of threading with any isometry mapping we prove that in a flat complete $\CAT(0)$ space the closure of the convex hull of a compact set is compact.
 We apply our theory to the computability of the Fr\'echet mean of a finite set of points and show that it is constructible in at most a finite number of steps, whenever the underlying space is of finite type. 
\end{abstract} 

\maketitle

\section{Introduction}\label{s:intro}
It is known that in a Hilbert space given a compact set the closure of its convex hull is compact \footnote{This is true for any locally convex topological vector space, see e.g. \cite[Theorem 5.35, p.185]{Border}.}. In finite dimensional Euclidean spaces even a stronger result holds that the convex hull itself of a compact set is compact, a conclusion that follows immediately from Carath\'eodory's Theorem. Here we investigate compact sets and the closure of their convex hulls in the setting of a complete $\CAT(0)$ space. $\CAT(0)$ spaces are geodesically connected metric spaces where every geodesic triangle is at least as {\em thin} as its comparison triangle in the Euclidean plane. This property
generalizes in a way the notion of nonpositive sectional curvature for a Riemannian manifold. Such a definition for curvature in a metric space is due to Alexandrov \cite{Alexandrov} and it was later popularized by Gromov \cite{Gromov2},  \cite{Gromov}. Often one refers to a $\CAT(0)$ space as a metric space of nonpositive curvature in the sense of Alexandrov (see \cite{Ballman}, \cite{Brid}). A Hilbert space is a particular example of a complete $\CAT(0)$ space. 

When a $\CAT(0)$ space is complete and locally compact then the closure of the convex hull of a compact set is always compact. This is a direct corollary of Hopf--Rinow Theorem which states that closed and bounded sets are compact whenever the underlying geodesic metric space is complete and locally compact. Indeed if a set is compact then it must be bounded and closed, thus contained in a closed geodesic ball of a certain radius. Because geodesic balls in $\CAT(0)$ spaces are convex then the closure of the convex hull of the given compact set must certainly be included in this ball and consequently be compact.
While the problem seems to be rather trivial for the locally compact case, it is not so when local compactness is removed as an assumption.  The question remains widely open even for the simplest case of a set of only three points as pointed out by Gromov \cite{Gromov}. Moreover in the locally compact case it is not clear whether the convex hull itself of a compact set, without taking its closure, is compact.

An elementary geometric interpretation of Carath\'eodory's Theorem (see Example \ref{example1}) and the geodesic structure of a $\CAT(0)$ space leads us to define a set operation that we call {\it{threading}}. This notion helps us derive sufficient conditions for when the closure of the convex hull of a compact set is compact. 
Essentially threading of a set $S\subseteq X$ is the union of all geodesic segments with endpoints in $S$. It is well behaved under the intersections and unions of sets and satisfies certain monotonicity properties. Moreover this operation preserves compactness. This has an important implication for our main problem. If the underlying space is of {\it finite type} then for every compact set its convex hull is compact. In relation to this topic Kope\v cka and Reich have made the following important observation that it suffices to consider the problem for finite sets only \cite[Theorem 2.10]{Reich} i.e. for any finite set $\{x_1,\cdots,x_n\}\subseteq X$, if $\cl\co\{x_1,\cdots,x_n\}$ is compact then for each compact set $S\subseteq X$ the closure of its convex hull $\cl\co S$ is compact. To verify this condition it is enough then to show that every finite set has a {\it finite threading degree}. Moreover using our theory of threading we are able to show that at least in the case of {\it flat} complete $\CAT(0)$ spaces the closure of the convex hull of any compact set is compact. This in turn solves the problem for a certain class of $\CAT(0)$ spaces. As a practical application of our theory we consider the computation of the Fr\'echet mean from a given finite set of points. We prove that in a $\CAT(0)$ space of finite type the Fr\'echet mean of a finite set $S$ lies in the convex hull $\co S$ and it is {\it constructible} from $S$ in at most a finite number of steps. The Fr\'echet mean is relevant in computational biology, in particular for computing the average phylogenetic tree from a given finite collection of phylogenetic trees see e.g. the seminal work of \cite{BHV}[Billera et al.] and more recently \cite{OwenLub}[Owen et al.].

Our work develops along the following lines. In Section \ref{s:preliminaries} we present some basic concepts about $\CAT(0)$ spaces and Painlev\'e--Kuratowski convergence. 
In Section \ref{s:characterization}
we state and prove characterization results about locally compact spaces in terms of compact sets and the closure of their convex hulls. In particular we prove that a complete $\CAT(0)$ space is locally compact if and only if the Painlev\'e--Kurtowski limit of any bounded nondecreasing net of compact sets is compact (Theorem \ref{th:regular}, Corollary \ref{c:regular}). Secondly we show that a complete $\CAT(0)$ space satisfying the geodesic extension property is locally compact if and only if it fulfills the so called {\it finite set property} and for every compact set the closure of its convex hull is compact (Theorem \ref{finiteset}). In Section \ref{s:threading} we introduce the operation of threading on a set. We provide a characterization result of the convex hull of a set in terms of its threadings of various degrees (Theorem \ref{thrcvxthm}). Moreover we prove that threading preserves compactness of a set (Theorem \ref{thrcompact}). By using the commutativity of threading with any isometry mapping we show that in a flat complete $\CAT(0)$ space the closure of the convex hull of a compact set is always compact (Theorem \ref{th:flat}). We also illustrate threading with a couple of examples. In Section \ref{s:Frechet} we apply our theory to computability of the Fr\'echet mean of a given finite set $S$. We show that the Fr\'echet mean always lies on the closure of the convex hull $\cl\co S$ (Lemma \ref{Frechet}) and that it is constructible from $S$ in a finite number of steps depending on the threading degree of $S$, whenever the underlying space is of finite type (Theorem \ref{complexity}).

\section{Preliminaries}
\label{s:preliminaries}
\subsection{$\CAT(0)$ spaces}
Let $(X,d)$ be a metric space. A {\em geodesic segment} starting from $x\in X$ and ending at $y\in X$ is a mapping $\gamma:[0,\ell]\to X$ such that $\gamma(0)=x,\gamma(\ell)=y$ and $d(\gamma(t_1),\gamma(t_2))= |t_1-t_2|$ for all $t_1,t_2\in [0,\ell]$. Often we denote this segment by $[x,y]$. A metric space $(X,d)$ is a (uniquely) geodesic metric space if every two elements $x,y\in X$ are connected by a (unique) geodesic segment. 
The collection of three elements $x,y,z\in X$ and the geodesic segments connecting them $[x,y], [y,z]$ and $[z,x]$ determines a geodesic triangle $\Delta(x,y,z)$. 
To every geodesic triangle corresponds a comparison triangle in the Euclidean plane, that is, three line segments $[\overline{x},\overline{y}], [\overline{y},\overline{z}]$ and $[\overline{z},\overline{x}]$ in $\mathbb R^2$, such that $d(x,y)=\|\overline{x}-\overline{y}\|, d(y,z)=\|\overline{y}-\overline{z}\|, d(z,x)=\|\overline{z}-\overline{x}\|$. A comparison point for $x'\in[x,y]$ is a point $\overline{x}'\in[\overline{x},\overline{y}]$ such that $d(x,x')=\|\overline{x}-\overline{x}'\|$. A geodesic metric space $(X,d)$ is a $\CAT(0)$ space if for every geodesic triangle $\Delta(x,y,z)$ and every $x'\in[x,y], y'\in[y,z]$ the inequality $d(x',y')\leq\|\overline{x}'-\overline{y}'\|$ holds true, where $\overline{x}'\in[\overline{x},\overline{y}], \overline{y}'\in[\overline{y},\overline{z}]$ are the comparison points of $x'$ and $y'$ respectively. 

A geodesic metric space $(X,d)$ is said to have the {\em geodesic extension} property if for every geodesic $\gamma:[0,\ell]\to X$ there exists $\varepsilon>0$ and a geodesic $\widetilde{\gamma}:[0,\widetilde \ell]\to X$ such that $\widetilde \ell=\ell+\varepsilon$ and 
$\widetilde{\gamma}\mid _{[0,\ell]}=\gamma$. 
In the particular case when $(X,d)$ is a $\CAT(0)$ space then geodesic extension property is equivalent to saying that any geodesic segment of positive length can be extended indefinitely to a geodesic line $\gamma:\mathbb{R}\to X$ \cite[Lemma 5.8 $(2)$]{Brid}. Examples of $\CAT(0)$ spaces satisfying geodesic extension property include but are not limited to Hilbert spaces, Hadamard manifolds, polyhedral complexes without free faces \cite[Proposition 5.10]{Brid}, and in general any $\CAT(0)$ space that is homeomorphic to a finite dimensional manifold \cite[Proposition 5.12]{Brid}.

Given $x\in X$ and $r>0$ we denote by $\mathbb B[x,r]\coloneqq \{y\in X\,:\,d(x,y)\leq r\}$ (a closed geodesic ball) and $\mathbb B(x,r)\coloneqq \{y\in X\,:\,d(x,y)< r\}$ (an open geodesic ball).
A set $C\subseteq X$ is convex if for any $x,y\in C$ the segment $[x,y]$ is entirely contained in $C$. Closed and open geodesic balls are examples of convex sets in a $\CAT(0)$ space.
Given a set $C\subseteq X$ we define $P_Cx\coloneqq\{y\in C\;:\; d(x,y)=d(x,C)\}$. When  $X$ is a complete $\CAT(0)$ space \footnote{Complete $\CAT(0)$ spaces are also known as {\em Hadamard spaces}.} and $C$ is a closed convex set then $P_Cx$ is nonempty and a singleton for every $x\in X$ (\cite[Proposition 2.4]{Berdellima}). Moreover the inequality holds
\begin{equation}
\label{eq:projectionsineq}
d(x,P_Cx)^2+d(P_Cx,y)^2\leq d(x,y)^2,\;\forall x\in X,\forall y\in C.
\end{equation}
Given $t\in[0,1]$ we let $x_t\coloneqq(1-t)x\oplus ty$ be the element on $[x,y]$ such that $d(x_t,x)=td(x,y)$.
An equivalent characterization of a $\CAT(0)$ space \cite[Definition 1.2.1]{Bacak} is given by 
\begin{equation}
\label{eq:quadratic}
d(x_t,z)^2\leq (1-t)d(x,z)^2+td(y,z)^2-t(1-t)d(x,y)^2,\;\forall x,y,z\in X,\forall t\in[0,1].
\end{equation}
From this inequality it can be shown that a $\CAT(0)$ space is uniquely geodesic. A $\CAT(0)$ space is called {\it flat} if inequality \eqref{eq:quadratic} holds with equality everywhere. A Hilbert space is a flat complete $\CAT(0)$ space.
A non-Hilbertian example is the Wasserstein space of probability measures on $\mathbb R$. For a geometric treatment of such spaces see e.g. \cite{Benoit}.

\subsection{Nets and Painlev\'e--Kuratowski convergence}

Let $(A,\preceq)$ be a directed set and $(X,d)$ a metric space equipped with the usual metric topology. A {\em net} $(x_{\alpha})_{\alpha\in A
}$ in $X$ is a mapping $\psi:A\to X$ ($\alpha\mapsto x_{\alpha}$).  We say a net $(x_{\alpha})_{\alpha\in A}$ in $X$ converges to an element $x$ in $X$ if for every neighborhood $U\subseteq X$ of $x$ there exists $\alpha_0\in A$ such that $x_{\alpha}\in U$ for every $\alpha\succeq \alpha_0$. The definition of nets can be extended to that of nets of sets. Let $\mathcal P(X)$ denote the set of all subsets in $X$. A {\em net of sets} $(S_{\alpha})_{\alpha\in A}$ in $\mathcal{P}(X)$ is a mapping $\psi:A\to \mathcal{P}(X)$ ($\alpha\mapsto S_{\alpha}$). We say a net of sets $(S_{\alpha})_{\alpha\in A}$ is nondecreasing (nonincreasing) if $S_{\alpha}\subseteq S_{\beta}$ ($S_{\beta}\subseteq S_{\alpha}$) whenever $\alpha\preceq\beta$.
A collection of sets $\{S_{\alpha}\}_{\alpha\in A}$ is a {\em chain} in $X$ if $S_{\alpha}\subseteq S_{\beta}$ or $S_{\beta}\subseteq S_{\alpha}$ whenever $\alpha\neq \beta$ i.e. $\{S_{\alpha}\}_{\alpha\in A}$ is a {\em totally ordered} subset of $\mathcal{P}(X)$. 
It is clear that a chain is a net of sets that is either nondecreasing or nonincreasing.
Given a net of sets $(S_{\alpha})_{\alpha\in A}$ in $X$ its Painlev\'e--Kuratowski outer and inner limit are defined respectively as 
\begin{equation}
\label{eq:upperlimit}
\Limsup_{\alpha}S_{\alpha}\coloneqq\{x\in H: \forall V\in\mathcal{N}(x), \forall \alpha_0\in A, \exists \alpha\succeq\alpha_0,\text{with}\; V\cap S_{\alpha}\neq\emptyset\}
\end{equation}
\begin{equation}
\label{eq:lowerlimit}
\Liminf_{\alpha}S_{\alpha}\coloneqq\{x\in H: \forall V\in\mathcal{N}(x), \exists \alpha_0\in A,\text{with}\; V\cap S_{\alpha}\neq\emptyset,\;\forall\alpha\succeq\alpha_0\}
\end{equation}
where $\mathcal{N}(x)$ is the collection of neighborhoods at $x$. If \eqref{eq:upperlimit} and \eqref{eq:lowerlimit} coincide then the limit of the net $(S_{\alpha})_{\alpha\in A}$ exists and we denote it by $\Lim_{\alpha}S_{\alpha}$. 
Note that when $(A,\preceq)=(\mathbb{N},\leqslant)$ then the above definition coincides with the usual definition of Painlev\'e--Kuratowski limit of a sequence of sets see e.g. \cite[Rockafellar and Wetts]{Rockafellar2}.  
It is known that in general any sequence of sets has a subsequence converging either to a nonempty set or the so called {\em horizon} (see \cite[Theorem 3.11]{Lukenotes}). However for our purposes we only need the following lemma concerning nondecreasing chains. 
\begin{lemma}
	\label{Painleve1}
	Let $X$ be a topological space and $(S_{\alpha})_{\alpha\in A}$ be a nondecreasing chain in $X$. Then $\Lim_{\alpha}S_{\alpha}=\cl\bigcup_{\alpha\in A}S_{\alpha}$.
\end{lemma}
\begin{proof}
	Let $x\in \cl\bigcup_{\alpha\in A}S_{\alpha}$ then for every $V\in\mathcal{N}(x)$ we have $V\cap\Big(\bigcup_{\alpha\in A} S_{\alpha}\Big)\neq\emptyset$. So there exists some $\alpha_0\in A$ such that $V\cap S_{\alpha_0}\neq\emptyset$. Define $N\coloneqq\{\alpha\in A:\alpha\succeq \alpha_0\}$. Assumption that $(S_{\alpha})_{\alpha\in A}$ is a nondecreasing chain implies that $V\cap S_{\alpha}\neq\emptyset$ for all $\alpha\in N$. By definition \eqref{eq:lowerlimit} we get $x\in\Liminf_{\alpha}S_{\alpha}$. On the other hand nondecreasing property of the chain $(S_{\alpha})_{\alpha\in A}$ implies $\Limsup_{\alpha}S_{\alpha}\subseteq\Liminf_{\alpha}S_{\alpha}$ hence $\Liminf_{\alpha}S_{\alpha}=\Limsup_{\alpha}S_{\alpha}$.
	By virtue of Painlev\'e --Kuratowski definition it follows that $x\in \Lim_{\alpha}S_{\alpha}$. Thus $\cl\bigcup_{\alpha\in A}S_{\alpha}\subseteq \Lim_{\alpha}S_{\alpha}$. Now let $x\in \Lim_{\alpha}S_{\alpha}$ then again by definition of Painlev\'e --Kuratowski limit we have $x\in\Liminf_{\alpha}S_{\alpha}$.
	This means that for all $V\in\mathcal{N}(x)$ we have $V\cap\bigcup_{\alpha\in A}S_{\alpha}\neq\emptyset$ implying $x\in \cl\bigcup_{\alpha\in A}S_{\alpha}$. Consequently $\Lim_{\alpha}S_{\alpha}\subseteq \cl\bigcup_{\alpha\in A}S_{\alpha}$. Therefore $\Lim_{\alpha}S_{\alpha}= \cl\bigcup_{\alpha\in A}S_{\alpha}$.
	%
\end{proof}

\section{Characterization results for locally compact spaces}
\label{s:characterization}

We say a complete geodesic metric space $(X,d)$ is regular if for any bounded nondecreasing chain of compact sets $(K_{\alpha})_{\alpha\in A}$ in $X$, $\Lim_{\alpha}K_{\alpha}$ is compact.

\begin{thm}
	\label{th:regular}A complete geodesic metric space $(X,d)$ is locally compact if and only if it is regular.
\end{thm}

\begin{proof}
	Let $(X,d)$ be a complete geodesic metric space that is locally compact and let $(A,\preceq)$ be some directed set. Let $(K_{\alpha})_{\alpha\in A}$ be a bounded nondecreasing chain of compact sets. By Lemma \ref{Painleve1} $\Lim_{\alpha}K_{\alpha}=\cl\bigcup_{\alpha\in A}K_{\alpha}$. Hence $\Lim_{\alpha}K_{\alpha}$ is a closed and bounded subset of $(X,d)$. Since $X$ is locally compact by Hopf--Rinow Theorem it follows that $\Lim_{\alpha}K_{\alpha}$ is compact. The nondecreasing chain of compact sets $(K_{\alpha})_{\alpha\in A}$ was arbitrary, thus $(X,d)$ is regular.
	
	Now we show the other direction. Suppose $(X,d)$ is regular.
	Let $x\in X$ be arbitrary and $\mathbb{B}[x,r]$ the closed geodesic ball centered at $x$ with radius $r>0$. 
	Denote by 
	$\mathcal{K}\coloneqq\{K\subseteq \mathbb{B}[x,r]: K\,\text{is compact}\}$.
	Clearly $\mathcal{K}$ is nonempty since any finite set of points in $\mathbb{B}[x,r])$ is a compact set. Let $(K_{\alpha})_{\alpha\in A}$ be a bounded nondecreasing chain in $\mathcal{K}$. By Lemma \ref{Painleve1} it follows that the limit $K^*\coloneqq\Lim_{\alpha}K_{\alpha}$ exists. 
	Assumption that $(X,d)$ is regular implies that $K^*$ is a compact set and hence $K^*\in\mathcal{K}$. Define a partial order on $\mathcal{K}$ by $K_1\preceq K_2$ if and only if $K_1\subseteq K_2$. For every chain $(K_{\beta})_{\beta\in B}$ in $\mathcal{K}$ we have $K_{\beta}\preceq K^*$, hence every chain is bounded in $\mathcal{K}$. Zorn's Lemma implies that $\mathcal{K}$ has at least one maximal element $K^{**}$. 
	We claim that $K^{**}$ concides with $\mathbb{B}[x,r]$. If not then there is an element $y\in\mathbb{B}[x,r]\setminus K^{**}$. Construct $K'_{\beta}\coloneqq
	K_{\beta}\cup\{y\}$ where $(K_{\beta})_{\beta\in B}$ is an arbitrary chain of compact sets in $\mathbb{B}[x,r]$. Clearly $y\in K'_{\beta}$ and $K'_{\beta}\in\mathcal{K}$ for all $\beta\in B$.
	Let $(K')^*\coloneqq\Lim_{\beta}K'_{\beta}$ then maximality of $K^{**}$ requires that $(K')^*\preceq K^{**}$. But this is impossible since by construction $y\in(K')^*$ and $y\notin K^{**}$. 
	Therefore $K^{**}=\mathbb{B}[x,r]$ which in turn yields that $\mathbb{B}[x,r]\in\mathcal{K}$ and hence $\mathbb{B}[x,r]$ is compact. For any $r'<r$ the open geodesic ball $\mathbb{B}(x,r')$ is entirely contained in $\mathbb{B}[x,r]$. Because $x\in X$ was arbitrary then $(X,d)$ must be locally compact.
\end{proof}

\begin{corollary}
	\label{c:regular}
	A complete $\CAT(0)$ space is locally compact if and only if it is regular.
\end{corollary}

\begin{lemma}
	\label{l:claims}Let $(X,d)$ be a metric space and $K\subseteq X$ be a compact convex set.  Let $\mathcal{F}_K$ be the collection of all finite subsets $S$ contained in $K$. Then 
	there is an increasing sequence of sets $S_1\subseteq S_2\subseteq...$ in $\mathcal{F}_K$ such that 
	\begin{equation}
	\label{claimK2}
	K=\cl\bigcup_{n\in\mathbb{N}}\co S_n.
	\end{equation}
\end{lemma}
\begin{proof}
	Let $K$ be a compact convex set. Then $K$ is separable, i.e. $K$ contains a countable dense subset $S\coloneqq\{x_1, x_2,..., x_n,...\}$. Let $S_n\coloneqq\{x_1, x_2,..., x_n\}$ then $(S_n)_{n\in\mathbb{N}}$ is an increasing sequence in $\mathcal{F}_K$. It follows that $\co S_n\subseteq K$ since $K$ is convex. This in turn yields 
	$$\cl\bigcup_{n\in\mathbb{N}}\co S_n\subseteq K.$$
	On the other hand we have the immediate inclusions
	$$S_i\subseteq \bigcup_{n\in\mathbb{N}}\co S_n,\quad \forall i\in\mathbb{N}$$
	implying 
	$$\cl S\subseteq \cl\bigcup_{n\in\mathbb{N}}\co S_n.$$
	Because $S$ is dense in $K$ then the equation
	$K=\cl S$ implies identity \eqref{claimK2}.
\end{proof}

For a given metric space $(X,d)$ and two bounded sets $A,B$ in $X$ let $d_H(A,B)$ denote the {\em Hausdorff distance} between $A$ and $B$ defined as 
\begin{equation}
\label{Hausdorffdist}
d_H(A,B)\coloneqq\max\{\sup_{y\in A}\inf_{x\in B}d(x,y), \sup_{x\in B}\inf_{y\in A}d(x,y)\}.
\end{equation}
It is known that $d_H(\cdot,\cdot)$ is a {\em pseudometric} on the set of bounded sets of a metric space $(X,d)$ (see \cite{Banas}). When $X$ is compact it is known that Painlev\'e--Kuratowski convergence coincides with Hausdorff distance convergence (see \cite{Kuratowski}, \cite[ Corollary 5.1.11]{Beer}).

We say a $\CAT(0)$ space $(X,d)$ satisfies the {\em finite set property} if for every $x\in X$ there exists an open set $U$ containing $x$ and a finite set $S$ such that $U\subseteq\cl\co S$.

\begin{thm}
	\label{finiteset}
	Let $(X,d)$ be a complete $\CAT(0)$ space satisfying the geodesic extension property. Then
	$(X,d)$ is locally compact if and only if $(X,d)$ satisfies the finite set property and for every compact set $K$ the closure of its convex hull $\cl\co K$ is compact.
\end{thm}

\begin{proof}
	Let $(X,d)$ be a locally compact complete $\CAT(0)$ space. If $K\subseteq X$ is a compact set then it is evident by arguments in Section \ref{s:intro} that $\cl\co K$ is compact. Now suppose that there is some $x\in X$ such that for any open set $U$ containing $x$ there is no finite set $S$ such that $U\subseteq\cl\co S$. Since $(X,d)$ is locally compact by definition there is an open set $V$ containing $x$ and a compact set $K$ such that $V\subseteq K$. Suppose $K$ is convex, else take $\cl\co K$ which is also compact by the first implication. Consider any finite $S$ in $K$ containing $x$. 
	For each such finite set $S$ one can find a sequence $(x_n)_{n\in\mathbb N}$ such that $x_n\in \mathbb{B}(x,\varepsilon_n)\setminus\cl\co S$ where $\mathbb{B}(x,\varepsilon_n)\subset V$ and $\lim_{n\to+\infty}\varepsilon_n=0$. Clearly $(x_n)_{n\in\mathbb{N}}\subset X\setminus\cl\co S$. Then $\lim_{n\to+\infty}x_n=x$ implies $x\in\cl(X\setminus\cl\co S)$. The obvious inclusion $x\in\cl\co S$ yields that $x$ is in the boundary of $\cl\co S$ for any finite set $S$ in $K$ containing $x$. Since $K$ is compact by Lemma \ref{l:claims} it follows that there is a collection of increasing sets $(S_n)_{n\in\mathbb{N}}$ such that 
	\begin{equation}
	\label{eq:identity}
	K=\cl\bigcup_{n\in\mathbb{N}}\co S_n
	\end{equation}
	Since $K$ is closed and convex the inclusions $\cl\co S_n\subseteq K$ for all $n\in\mathbb{N}$ imply 
	\begin{equation}
	\label{eq:equivalent}
	K=\cl\bigcup_{n\in\mathbb{N}}\cl\co S_n.
	\end{equation}
	Assume without loss of generality that $x\in S_n$ for all $n\in\mathbb{N}$, else we can always add $x$ to $S_n$ and obtain a new sequence of increasing sets satisfying identities \eqref{eq:identity} and \eqref{eq:equivalent}. By above arguments it follows that $x$ is in the boundary of each $\cl\co S_n$. Moreover in view of Lemma \ref{Painleve1} we have $\lim_{n\to+\infty}\cl\co S_n=K$ and subsequently $\lim_{n\to+\infty}d_H(K,\cl\co S_n)=0$ since $K$ is compact.
	The inclusion $\cl\co S_n\subseteq K$ and definition \eqref{Hausdorffdist} yield 
	$$d_H(K,\cl\co S_n)=\sup_{y\in K}\inf_{z\in\cl\co S_n}d(z,y).$$
	Because $\cl\co S_n$ is a closed convex set then for each $y\in K$ its projection $P_{\cl\co S_n}y$ onto $\cl\co S_n$ exists and it is unique, therefore we obtain
	$$d_H(K,\cl\co S_n)=\sup_{y\in K}d(y,P_{\cl\co S_n}y).$$
	Let $\partial K$ denote the boundary of $K$. Then $\partial K\subseteq K$ implies
	$$\sup_{y\in K}d(y,P_{\cl\co S_n}y)\geq \sup_{y\in \partial K}d(y,P_{\cl\co S_n}y)$$ and thus
	\begin{equation}
	\label{eq:technical}
	d_H(K,\cl\co S_n)\geq \sup_{y\in \partial K}d(y,P_{\cl\co S_n}y).
	\end{equation}
	There exists a sequence $(y_k)_{k\in\mathbb N}\subseteq K\setminus\{x\}$ such that $\lim_{k\to+\infty}y_k=x$. Let $P_{\cl\co S_n}y_k$ denote the metric projection of $y_k$ onto $\cl\co S_n$ for every $k\in\mathbb{N}$. Denote by $\gamma_k:[0,1]\to X$ the geodesic segment connecting $P_{\cl\co S_n}y_k$ with $y_k$. By assumption $(X,d)$ satisfies the geodesic extension property. Therefore there exists geodesic lines $\widetilde{\gamma}_k:\mathbb{R}\to X$ such that $\widetilde{\gamma}_k\mid_{[0,1]}=\gamma_k$ for every $k\in\mathbb{N}$. Since $K$ is bounded and the image of $\widetilde{\gamma}_k$ is connected then there exists $z_k\in \widetilde{\gamma}_k\cap\partial K$ for every $k\in\mathbb{N}$. From the equation $P_{\cl\co S_n}z_k=P_{\cl\co S_n}y_k$ for all $k\in\mathbb{N}$ we obtain the inequalities
	$$d(z_k,y)^2\geq d(z_k,P_{\cl\co S_n}z_k)^2+d(P_{\cl\co S_n}z_k,y)^2,\quad\forall k\in\mathbb{N},\; \forall y\in\cl\co S_n.$$
	Note that $$\lim_{k\to+\infty}d(P_{\cl\co S_n}z_k,x)=\lim_{k\to+\infty} d(P_{\cl\co S_n}y_k,x)\leq \lim_{k\to+\infty}d(y_k,x)=0$$ implies that $\lim_{k\to+\infty}P_{\cl\co S_n}z_k=x$.
	Since $(z_k)_{k\in\mathbb{N}}\subseteq \partial K$ and $\partial K$ is compact then there is a subsequence $(z_{k_m})_{m\in\mathbb{N}}\subseteq(z_k)_{k\in\mathbb{N}}$ converging to some element $z\in \partial K$. Passing in the limit we obtain 
	\begin{align*}
	\lim_{m\to+\infty}d(z_{k_m},y)^2&\geq \lim_{m\to+\infty}d(z_{k_m},P_{\cl\co S_n}z_{k_m})^2+\lim_{m\to+\infty}d(P_{\cl\co S_n}z_{k_m},y)^2\\&\Rightarrow d(z,y)^2\geq d(z,x)^2+d(x,y)^2,\quad\forall y\in\cl\co S_n. 
	\end{align*}
	In particular when $y=P_{\cl\co S_n}z$ we get 
	$$d(z,P_{\cl\co S_n}z)^2\geq d(z,x)^2+d(x,P_{\cl\co S_n}z)^2.$$
	On the other hand by the property of the projection onto a closed convex set we have the inequality
	$$d(z,x)^2\geq d(z,P_{\cl\co S_n}z)^2+d(x,P_{\cl\co S_n}z)^2.$$
	Together these last two inequalities imply $d(x,P_{\cl\co S_n}z)\leq 0$ and hence $x=P_{\cl\co S_n}z$. From \eqref{eq:technical} it follows then 
	\begin{align*}
	d_H(K,\cl\co S_n)&\geq \sup_{y\in \partial K}d(y,P_{\cl\co S_n}y)\\&\geq d(z,P_{\cl\co S_n}z)=d(z,x)\geq \inf_{u\in\partial K}d(u,x)=d(x,\partial K).
	\end{align*}
	In the limit we obtain $\lim_{n\to+\infty}d(x,\partial K)\leq \lim_{n\to+\infty}d_H(K,\cl\co S_n)=0$ or equivalently $x\in\partial K$. However this is impossible since $x\in\inter K$.
	
	Now we show the opposite direction.
	Suppose that for every compact set $K$ the closure of its convex hull $\cl\co K$ is compact and that $(X,d)$ satisfies the finite set property. For every $x\in X$ there is an open set $U$ containing $x$ and a finite set $S$ such that $U\subseteq \cl\co S$. But $S$ is finite and therefore compact which in turn yields that $\cl\co S$ is compact. By definition of local compactness it follows that $(X,d)$ is a locally compact space.
\end{proof}

\section{Threading of a set}
\label{s:threading}

\subsection{Main properties of threading}
For $S\subseteq X$ the \textit{threading} of $S$ denoted by $\thr S$ is defined to be the union of all geodesic segments with both endpoints in $S$.  
It follows from this definition that if $z\in \thr S$ then there are $x,y\in S$ such that $z\in[x,y]$. In general $[x',y']\subseteq \thr S$ if and only if there are $x,y\in S$ such that $[x',y']\subseteq[x,y]$. One can then iteratively define threading of threading of a set and so on. Given a set $S$ we have the following chain of inclusions $S\subseteq \thr S\subseteq\thr^2S\subseteq...\subseteq\thr^nS\subseteq...$ and the equation $\thr^nS=\thr(\thr^{n-1}S)$ for all $n\in\mathbb{N}$ is immediate. 
\begin{proposition}
	\label{thrproperties}
	In general the following rules hold
	\begin{align}
	&\thr(S_1\cap S_2)\subseteq\thr S_1\cap \thr S_2\label{eq:id1}\\
	&\thr S_1\cup\thr S_2\subseteq\thr(S_1\cup S_2)\label{eq:id2}
	\end{align}
	with equality in both if $S_1\subseteq S_2$ or $S_2\subseteq S_1$.
\end{proposition}
\begin{proof}
	Let $z\in \thr(S_1\cap S_2)$ then by definition of threading we have $x,y\in S_1\cap S_2$ such that $z\in [x,y]$. But $x,y\in S_1\cap S_2$ implies $x,y\in S_1$ and $x,y\in S_2$. Therefore $[x,y]\subseteq \thr S_1$ and $[x,y]\subseteq\thr S_2$ and so $z\in\thr S_1\cap\thr S_2$. Then
	identity \eqref{eq:id1} follows. Now let $z\in\thr S_1\cup \thr S_2$ then $z\in \thr S_1$ or $z\in \thr S_2$. There are $x_1,y_1\in S_1$ and $x_2,y_2\in S_2$ such that $z\in[x_1,y_1]$ or $z\in[x_2,y_2]$. By construction since $x_1,y_1,x_2,y_2\in S_1\cup S_2$ then $[x_1,y_1],[x_2,y_2]\subseteq \thr(S_1\cup S_2)$ implies $z\in \thr(S_1\cup S_2)$.
	This proves identity \eqref{eq:id2}. 
	If $S_1\subseteq S_2$ or $S_2\subseteq S_1$ the equalities are evident in both identities. 
\end{proof}
In general one can find sets $S_1, S_2$ such that \eqref{eq:id1} or \eqref{eq:id2} hold with strict inclusion. Consider the simplest example of $X=\mathbb{R}$ equipped with the usual metric $d(r,s)=|r-s|$ for any two real numbers $r,s\in\mathbb{R}$. Take $S_1\coloneqq\{0,2\}, S_2=\{1\}$. Then $S_1\cap S_2=\emptyset$ implies 
$\thr(S_1\cap S_2)=\emptyset\subsetneq \{1\}=[0,2]\cap\{1\}=\thr S_1\cap\thr S_2$.  Similarly for the other identity let $S_2=\{3\}$ then $\thr S_1\cup \thr S_2=[0,2]\cup\{3\}\subsetneq[0,3]=\thr(S_1\cup S_2)$. Moreover from relation \eqref{eq:id1} it follows that $\thr S_1\subseteq\thr S_2$ whenever $S_1\subseteq S_2$. Indeed $S_1\subseteq S_2$ implies that $\thr S_1=\thr(S_1\cap S_2)\subseteq\thr S_1\cap \thr S_2\subseteq\thr S_2$.

\begin{thm}
	\label{thrcvxthm}
	If $S\subseteq X$ is convex then $\thr S=S$. In general for any set $S\subseteq X$ we have the following identity
	\begin{equation}
	\label{eq:cvxid}   
	\bigcup_{n\in\mathbb{N}}\thr^nS=\co S.
	\end{equation}
\end{thm}
\begin{proof}
	The inclusion $S\subseteq\thr S$ is clear whether $S$ is convex or not. Now let $z\in \thr S$. By definition of threading 
	there are $x,y\in S$ such that $z\in[x,y]$. Assumption $S$ is convex implies $[x,y]\subseteq S$ therefore $z\in S$ which in turn yields $\thr S\subseteq S$. To prove identity \eqref{eq:cvxid} let $x,y\in \bigcup_{n\in\mathbb{N}}\thr^nS$ then there are $l,m\in\mathbb{N}$ such that $x\in\thr^lS,y\in\thr^mS$. Assume that $m\geq l$ then $\thr^lS\subseteq\thr^mS$ implies $x\in\thr^mS$. Therefore $x,y\in\thr^mS$ yields by definition of threading $[x,y]\subseteq\thr^{m+1}S\subseteq\bigcup_{n\in\mathbb{N}}\thr^nS$. Hence $\bigcup_{n\in\mathbb{N}}\thr^nS$ is a convex set. But $\co S$ is the smallest convex set containing $S$ thus 
	$$\co S\subseteq \bigcup_{n\in\mathbb{N}}\thr^nS.$$
	On the other hand from \eqref{eq:id1} and the first part of this theorem we have 
	$$\thr S=\thr(S\cap\co S)\subseteq \thr S\cap \thr\co S\subseteq\thr\co S=\co S$$
	therefore iterative threading implies 
	$$\thr^nS\subseteq\co S,\hspace{0.2cm}\forall n\in\mathbb{N}\Rightarrow \bigcup_{n\in\mathbb{N}}\thr^nS\subseteq\co S.$$
\end{proof}

\begin{proposition}
	\label{p:isomorphism}
	Let $(X_1,d_1)$ and $(X_2,d_2)$ be complete $\CAT(0)$ spaces and $\Phi:X_1\to X_2$ an isometry. Then $\Phi(\thr S)=\thr\Phi(S)$ and in particular $\Phi(\cl\co S)=\cl\co\Phi(S)$ for any $S\subseteq X_1$. 
\end{proposition} 
\begin{proof}
	Let $x,y\in S$ then $[x,y]\subseteq \thr S$. Denote by $x_t\coloneqq(1-t)x\oplus ty$ for $t\in[0,1]$ then $$d_2(\Phi(x),\Phi(x_t))=d_1(x,x_t)=td_1(x,y)=td_2(\Phi(x),\Phi(y))$$
	implies $\Phi(x_t)=(1-t)\Phi(x)\oplus t\Phi(y)$ for all $t\in[0,1]$. Hence $\Phi([x,y])=[\Phi(x),\Phi(y)]$. This shows $\Phi(\thr S)\subseteq \thr\Phi(S)$. For the other direction use the inverse mapping $\Phi^{-1}$ which is as well an isometry. The same arguments yield $\thr\Phi(S)\subseteq \Phi(\thr S)$. In general by induction we obtain $\Phi(\thr^nS)=\thr^n\Phi(S)$ for every $n\in\mathbb N$. Since $\thr^nS\subseteq\thr^{n+1}S$ then we can equivalently state that $\Phi(\bigcup_{k=1}^n\thr^kS)=\bigcup_{k=1}^n\thr^k\Phi(S)$ for every $n\in\mathbb N$ and consequently by Theorem \ref{thrcvxthm} we obtain $\Phi(\co S)=\co\Phi(S)$. Taking the closure on both sides and using continuity of $\Phi$, since it is an isometry, we get that $\Phi(\cl\co S)=\cl\co\Phi(S)$. 
\end{proof}

For a given set $S\subseteq X$ the \textit{threading degree} of $S$ denoted by $\deg_{\thr}S$ is defined as the smallest $n\in\mathbb{N}$, if it exists, such that $\thr^nS=\co S$. 
The threading degree of a complete $\CAT(0)$ space $X$ is then the supremum of threading degrees of all subsets of $X$ i.e. 
\begin{equation}
\label{eq:thrdegree}
\deg_{\thr}X\coloneqq\sup_{S\subseteq X}\deg_{\thr}S.
\end{equation}
If $\deg_{\thr}X$ is finite we say that $X$ is of \textit{finite type}.

\begin{proposition}
	Let $(X,d)$ be a complete $\CAT(0)$ space of finite type. For any closed convex set $S\subseteq X$ the \textit{subspace} $(S,d_S)$ is of finite type, where $d_S\coloneqq d\mid_S$.
\end{proposition}

\begin{proof}
	First note that $(S,d_S)$ is itself a complete $\CAT(0)$ space. Let $\thr_SS'$ and $\thr_XS'$ denote threading applied to $S'\subseteq S$ in $S$ and $X$ respectively. Since $d_S=d\mid_S$ 
	then $\thr_SS'=\thr_XS'$. By definition \eqref{eq:thrdegree} it follows
	$$\deg_{\thr_S}S=\sup_{S'\subseteq S}\deg_{\thr_S}S'=\sup_{S'\subseteq S}\deg_{\thr_X}S'\leq \sup_{S'\subseteq X}\deg_{\thr_X}S'=\deg_{\thr_X}X.$$
	By assumption $\deg_{\thr_X}X<+\infty$ we get that $(S,d_S)$ is of finite type.
\end{proof}

\begin{proposition}
	\label{thrproduct}
	If $(X_1,d_1), (X_2,d_2)$ are of finite type then $(X,d)$, where $X\coloneqq X_1\times X_2$ and $d(\cdot,\cdot)$ is its canonical metric, is of finite type. Moreover the following holds $$\deg_{\thr}X=\max\{\deg_{\thr}X_1,\deg_{\thr}X_2\}.$$ 
\end{proposition}
\begin{proof}
	Let $S\subseteq X$ then $S=S_1\times S_2$ for some $S_1\subseteq X_1, S_2\subseteq X_2$. Let $z\in \thr S$ then there are $x,y\in S$ such that $z\in[x,y]$. On the other hand $x=(x_1,x_2)$ and $y=(y_1,y_2)$ for some $x_1,y_1\in S_1$ and $x_2,y_2\in S_2$. By construction if follows $[x,y]=[x_1,y_1]\times[x_2,y_2]$ therefore $z=(z_1,z_2)$ for some $z_1\in[x_1,y_1]$ and $z_2\in[x_2,y_2]$. This means that $z\in\thr S_1\times \thr S_2$. 
	The other direction is proved analogously. In general we get that $\thr^nS=\thr^nS_1\times\thr^nS_2$ for any $n\in\mathbb{N}$. It follows then that $\deg_{\thr}S=\max\{\deg_{\thr}S_1,\deg_{\thr}S_2\}$. On the other hand we have the identity $\{S:S\subseteq X\}=\{S_1:S_1\subseteq X_2\}\times\{S_2:S_2\subseteq X_2\}$.
	Taking supremum over $S\subseteq X$ 
	\begin{align*}
	\sup_{S\subseteq X}\deg_{\thr}S&=\sup_{S_1\subseteq X_1, S_2\subseteq X_2}\max\{\deg_{\thr}S_1,\deg_{\thr}S_2\}\\&=\max\{\sup_{S_1\subseteq X_1}\deg_{\thr}S_1,\sup_{S_2\subseteq X_2}\deg_{\thr}S_2\}
	\end{align*}
	which is equivalent to 
	$\deg_{\thr}X=\max\{\deg_{\thr}X_1,\deg_{\thr}X_2\}$.
\end{proof}

Notice that in general if $S_1,S_2\subseteq X$ are such that $S_1\subseteq S_2$ then neither $\deg_{\thr}S_1\leq\deg_{\thr}S_2$ nor $\deg_{\thr}S_2\leq\deg_{\thr}S_1$ is necessarily true. Consider $X=\mathbb{R}^2$ equipped with the usual Euclidean metric and let $S_1$ be four corners of a square and $S_2$ be the square itself. Clearly $S_1\subseteq S_2$ but $\deg_{\thr}S_1=2>1=\deg_{\thr}S_1$. Similarly one can show that if instead $S_1$ is just two of the corner points and $S_2$ is the four corner points then $\deg_{\thr}S_1=1<2=\deg_{\thr}S_2$. Figure \ref{fig:threading} depicts an example of threading for three distinct non-collinear points in the Euclidean plane $\mathbb{R}^2$.
\begin{figure}[t]
	\centering
	\includegraphics[width=11cm]{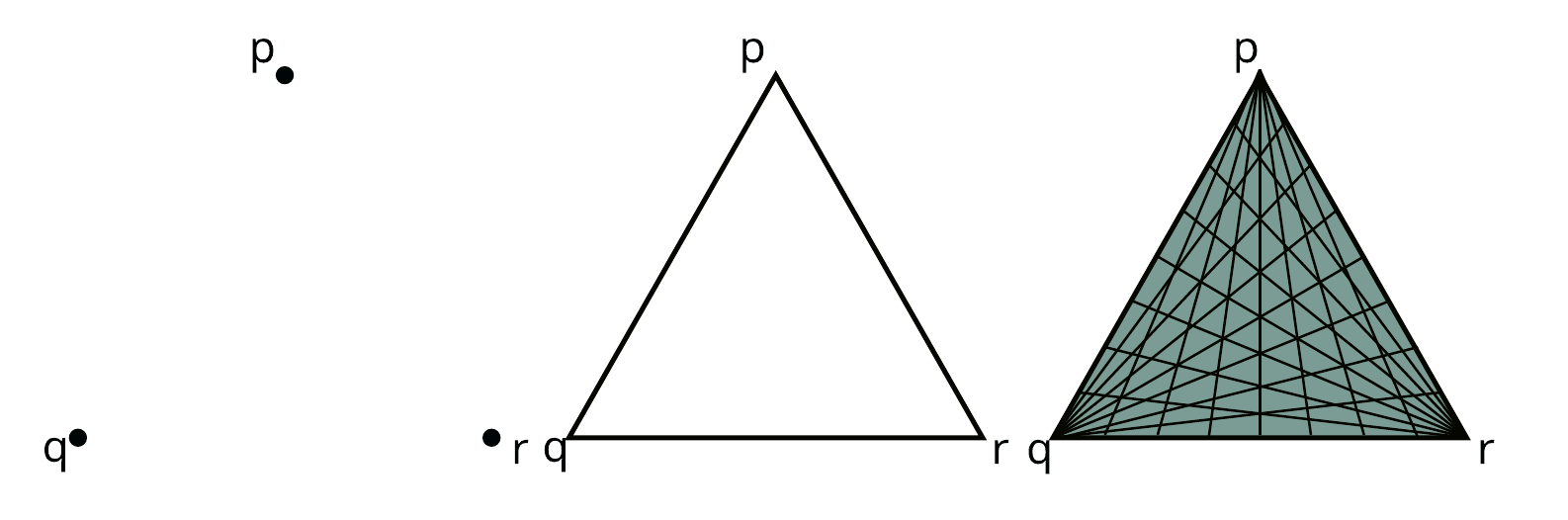}
	\caption{Three points $\{p,q,r\}$ in $\mathbb{R}^2$ (left), $\thr\{p,q,r\}$ is the Euclidean triangle (middle) and $\thr^2\{p,q,r\}$ is the solid Euclidean triangle (right).} \label{fig:threading}
\end{figure}

\subsection{Threading of compact sets}
\begin{thm}
	\label{thrcompact}
	$\thr^nK$ is compact for all $n\in\mathbb{N}$ whenever $K\subseteq X$ is compact.
\end{thm}
\begin{proof}
	First we prove that $\thr K$ is compact. Let $(x_k)_{k\in\mathbb{N}}\subseteq \thr K$ be some sequence. There are $y_k,z_k\in K$ and $t_k\in[0,1]$ such that $x_k=(1-t_k)y_k\oplus t_k\,z_k$ for all $k$. By assumption $K$ is compact then so is the product space $K\times K\times [0,1]$.
	There are convergent subsequences $(y_{k_m})_{m\in\mathbb{N}}, (z_{k_m})_{m\in\mathbb{N}}\subseteq K$ and $(t_{k_m})_{m\in\mathbb{N}}\subseteq [0,1]$. Let $\lim_{m\to+\infty}y_{k_m}=y, \lim_{m\to+\infty}z_{k_m}=z$ for $y,z\in K$ and $\lim_{m\to+\infty}t_{k_m}=t\in[0,1]$. Denote by $x\coloneqq (1-t)y\oplus t\,z$. Then we have from characterization inequality \eqref{eq:quadratic} for all $m\in\mathbb N$ that
	\begin{align*}
	d(x,x_{k_m})^2&\leq (1-t_{k_m})d(x,y_{k_m})^2+t_{k_m}d(x,z_{k_m})^2-t_{k_m}(1-t_{k_m})d(y_{k_m},z_{k_m})^2.
	\end{align*}
	Passing in the limit we obtain 
	\begin{align*}
	\lim_{m\to+\infty}\Big[(1-t_{k_m})&d(x,y_{k_m})^2+t_{k_m}d(x,z_{k_m})^2-t_{k_m}(1-t_{k_m})d(y_{k_m},z_{k_m})^2\Big]\\&=
	(1-t)d(x,y)^2+td(x,z)^2-t(1-t)d(y,z)^2\\&
	=(1-t)t^2d(z,y)^2+t(1-t)^2d(y,z)^2-t(1-t)d(y,z^2)=0.
	\end{align*}
	Therefore $\lim_{m\to+\infty}x_{k_m}=x\in \thr K$. Since the sequence $(x_k)_{k\in\mathbb N}\subseteq \thr K$ is arbitrary then this means that 
	$\thr K$ is sequentially compact and therefore compact. Now suppose $\thr^{n-1}K$ is compact. From the equation $\thr^nK=\thr(\thr^{n-1}K)$ we get that $\thr^nK$ is compact for any $n\in\mathbb{N}$.
\end{proof}

As immediate corollaries of the last result we obtain:

\begin{corollary}
	\label{thrdegcompact}
	The followings are true:
	\begin{enumerate}[(i)]
		\item Let $K\subseteq X$ be compact. If $\deg_{\thr}K$ is finite then $\co K$ is compact.
		\item 	If $X$ is of finite type then $\co K$ is compact whenever $K$ is compact.
	\end{enumerate}
	
\end{corollary}

From Kope\v cka--Reich observation \cite[Theorem 2.10]{Reich} we get the result:

\begin{corollary}
	Let $(X,d)$ be a complete $\CAT(0)$ space such that $\deg_{\thr}S$ is finite for any finite set $S\subseteq X$. Then the closure of the convex hull of a compact set is compact.
\end{corollary}
%


Note that a Hilbert space $\mathcal H$ fails to be of finite type. If $(e_n)_{n\in\mathbb N}$ is the standard basis then $\deg_{\thr}\{e_1,...,e_n\}=n$ (follows from Caratheodory's Theorem see Example \ref{example1}) implies 
$\deg_{\thr}\mathcal H\geq \sup_{n\in\mathbb N}\deg_{\thr}\{e_1,...,e_n\}=+\infty$.
Now consider the unbounded ray of real numbers with at point $x=n$ an $n$-dimensional cube attached (1-point union). In this example every compact set is of finite type but their threading degrees are not uniformly bounded. Consequently the space itself is not of finite type \footnote{These examples were pointed out by Prof. Thomas Schick.}.

\begin{thm}
	For a given set $K\subseteq X$ define
	\begin{equation}
	\label{eq:YK}
	Y_K\coloneqq\{(y_m)_{m\in\mathbb{N}}\subseteq \co K: \exists n\in\mathbb{N},\quad y_m\in\thr^nK\quad\text{for infinitely many}\quad m\in\mathbb{N}\}.
	\end{equation}
	If $K$ is compact and $\cl Y_K=\cl\co K$ then $\cl\co K$ is compact.
\end{thm}
\begin{proof}
	Let $K\subseteq X$ be a compact set and $Y_K$ be defined as in \eqref{eq:YK}. 
	Consider the limit $\lim_{n\to+\infty}\thr^nK$. By construction $\thr^{n-1}K\subseteq\thr^n K$ for any $n\in\mathbb{N}$. By Theorem \ref{thrcvxthm} we have $\co K=\bigcup_{n\in\mathbb{N}}\thr^nK$. Since $K\subseteq\thr K\subseteq...\subseteq \thr^{n-1}K\subseteq \thr^nK\subseteq...$ is a sequence of nested nondecreasing sets then by Lemma \ref{Painleve1} $\lim_{n\to+\infty}\thr^nK=\cl\bigcup_{n\in\mathbb{N}}\thr^nK$ or equivalently $\lim_{n\to+\infty}\thr^nK=\cl\co K$. In particular we obtain that when $K$ is compact $\cl\co K$ can be successively approximated in the sense of Painlev\'e--Kuratowski by a sequence of nondecreasing compact sets. Let $(x_m)_{m\in\mathbb{N}}\subseteq \cl\co K$ then for each $x_m$ there exists $y_m\in\co K$ such that $d(x_m,y_m)<1/m$. 
	Let $(y_m)_{m\in\mathbb N}\in Y_K$ else we can always find a sequence $(\widetilde{y}_m)_{m\in\mathbb N}\in Y_K$ such that $d(y_m,\widetilde{y}_m)<1/m$ for every $m\in \mathbb N$ because by assumption $\cl Y_K=\cl\co K$. Let $n\in\mathbb{N}$ be such that $\thr^nK$ contains infinitely many terms of the sequence $(y_m)_{m\in\mathbb{N}}$. 
	Denote this subsequence by $(y_{m_k})_{k\in\mathbb{N}}$.  By Lemma \ref{thrcompact} $\thr^{n}K$ is compact then there is some convergent subsequence $(y_{m_{k_j}})_{j\in\mathbb{N}}$. Let $\lim_jy_{m_{k_j}}=y$. The estimate
	$$d(x_{m_{k_j}},y)\leq d(x_{m_{k_j}},y_{m_{k_j}})+d(y_{m_{k_j}},y)$$
	implies $\lim_jx_{m_{k_j}}=y$. Therefore $(x_{m_{k_j}})_{j\in\mathbb{N}}$ would be a convergent subsequence of the original sequence $(x_m)_{m\in\mathbb{N}}$. Since $(x_m)_{m\in\mathbb N}$ is arbitrary then this means that $\cl\co K$ is sequentially compact and hence compact. 
\end{proof}

For the next result we make use of \cite[Theorem 7.2]{Jacob} that we state as a technical lemma.
\begin{lemma}
	\label{l:isometry}
	Let $(X,d)$ be a flat complete $\CAT(0)$ space. Then $X$ is isometric to a nonempty closed convex subset of a Hilbert space.
\end{lemma}

\begin{thm}
	\label{th:flat}
	Let $(X,d)$ be a flat complete $\CAT(0)$ space and $K\subseteq X$ be a compact set then $\cl\co K$ is compact. If additionally $K$ is finite then $\co K$ is compact. 
\end{thm}

\begin{proof}
	Let $(X,d)$ be a flat complete $\CAT(0)$ space and $K\subseteq X$ a compact set. By Lemma \ref{l:isometry} there exists a mapping $\Phi:X\to \mathcal H$ that is an isometry from $X$ into a nonempty closed convex subset of a Hilbert space $\mathcal H$. Since $K$ is compact then $\Phi(K)$ is a compact set in $\mathcal H$. Then it follows that $\cl\co \Phi(K)$ is compact. By Proposition \ref{p:isomorphism}	we have $\Phi(\cl\co K)=\cl\co \Phi(K)$. 
	Consider the inverse mapping $\Phi^{-1}:\Phi(X)\to X$, which itself is an isometry and in particular continuous. Then $$\cl\co K=\Phi^{-1}(\Phi(\cl\co K))=\Phi^{-1}(\cl\co \Phi(K))$$
	is a compact set. 
	Now assume that additionally $K$ is a finite set. Then $\Phi(K)$ is a finite set in $\mathcal H$ and so it lies in a finite dimensional subspace of $\mathcal H$. By virtue of Carath\'eodory's Theorem (e.g. see \cite{Rockafellar}[Theorem 17.1]) it follows that $\co\Phi(K)$ is compact. Again by Proposition \ref{p:isomorphism} we have $\Phi(\co K)=\co\Phi(K)$ and consequently $\co K=\Phi^{-1}(\Phi(\co K))=\Phi^{-1}(\co\Phi(K))$ is compact.
	This completes the proof.
\end{proof}

\subsection{Some illustrations}
\begin{example}[Euclidean spaces]
	\label{example1}\normalfont{
		Euclidean spaces are the simplest examples of complete $\CAT(0)$ spaces. Let $\mathbb{E}^d$ be a $d$-dimensional Euclidean space. By the well known Carath\'eodory Theorem if $x\in \co S$ for some $S\subseteq \mathbb{E}^d$ then $x$ can be expressed as a convex combination of at most $d+1$ points from $S$. This means that there are $s_1,...,s_{d+1}\in S$ such that $x\in\co\{s_1,...,s_{d+1}\}$. Algebraically we have the representation $x=a_1s_1+...+a_{d+1}s_{d+1}$ where $a_i\in[0,1]$ for all $i=1,...,d+1$ and $\sum_ia_i=1$. By letting $a'\coloneqq\sum_{i=1}^{d}a_i$ then we can rewrite $x=a_{d+1}s_{d+1}+(1-a_{d+1})\sum_{i=1}^d a'_is_i$ where $a'_i\coloneqq a_{i}/a'$ for all $i=1,2,...,d$. Hence $x'\coloneqq\sum_{i=1}^d a'_is_i$ is some point lying in $\co\{s_1,...,s_d\}$. Following iteratively this method one obtains a sequence of points $x^{(1)}, x^{(2)}, x^{(3)},...,x^{(d)}$, where $x^{(1)}=x$, such that $x^{(1)}\in\co\{s_1,...,s_{d+1}\}, x^{(2)}\in\co\{s_1,...,s_d\},...,x^{(d)}\in\co\{s_1, s_2\}$. By construction we have the following inclusions $$x^{(d)}\in\thr\{s_1, s_2\},...,x^{(2)}\in\thr^{d-1}\{s_1,...,s_{d}\},x^{(1)}\in\thr^{d}\{s_1,...,s_{d+1}\}$$ By virtue of Carath\'eodory Theorem this means that $\co\{s_1,...,s_{d+1}\}\subseteq \thr^{d}\{s_1,...,s_{d+1}\}$. Because $x\in S$ is arbitrary then $\co S\subseteq \thr^{d}S$ and so $\deg_{\thr}S\leq d$. On the other hand $S\subseteq\mathbb{E}^n$ is any subset hence $\deg_{\thr}\mathbb{E}^d\leq d$. Therefore any Euclidean space is of finite type.
		In fact it holds $\deg_{\thr}\mathbb{E}^d\leq d$. To see this take $d+1$-linearly independent points $\{s_1,...,s_{d+1}\}$ in $\mathbb E^d$ then if $x\in\thr^d\{s_1,...,s_d\}$ then it can be easily checked that there are real numbers not all zero $a_1,...,a_{d+1}$ such that $x=a_1s_1+...a_{d+1}s_{d+1}$ i.e. $x\in\co\{s_1,...,s_{d+1}\}$. By earlier arguments this means that $\co\{x_1,...,x_{d+1}\}=\thr^d\{x_1,...,x_{d+1}\}$. Thus $\deg_{\thr}\mathbb E^d=d$.}
\end{example}

\begin{example}
	[A non-Euclidean metric space] 
	\label{example2}
	\normalfont{Consider the set $X\coloneqq\mathbb{R}^2_{+}\cup\mathbb{R}^2_{-}$ where $\mathbb{R}^2_{+}\coloneqq\{x\in\mathbb{R}^2: x_1, x_2\geq 0\}$ and $\mathbb{R}^2_{-}\coloneqq\{x\in\mathbb{R}^2: x_1, x_2\leq 0\}$. If $X$ is equipped with a metric $d(\cdot,\cdot)$ induced from the length of the shortest path connecting any two points in $X$ then it can be shown that $(X,d)$ is a complete $\CAT(0)$ space. Now let $S\subseteq X$. If $S$ is contained entirely in one of the two quadrants of $X$ i.e. $S\subseteq\mathbb{R}^2_{+}$ or $S\subseteq\mathbb{R}^2_{-}$, then it is easily seen from the previous example that $\deg_{\thr}S\leq 2$. Now suppose that $S=S_1\cup S_2$ where $S_1=S\cap\mathbb{R}^2_{+}$ and $S_2=S\cap \mathbb{R}^2_{-}$. Notice that any point in $S_1$ can be joined with a point in $S_2$ by a shortest line (possibly broken) only passing through the origin $0\in\mathbb{R}^2$ since $\mathbb{R}^2_{+}\cap\mathbb{R}^2_{-}=\{0\}$. Therefore $\thr(S_1\cup \{0\})\cup\thr(S_2\cup\{0\})=\thr S$. By same arguments $\thr(S_1\cup \{0\})\cap\thr(S_2\cup\{0\})=\{0\}$ implies 
		$$\thr^2(S_1\cup \{0\})\cup\thr^2(S_2\cup\{0\}))=\thr(\thr(S_1\cup \{0\})\cup\thr(S_2\cup\{0\}))=\thr^2S$$
		and in more generality 
		\begin{equation}
		\label{eq:threq}
		\thr^n(S_1\cup \{0\})\cup\thr^n(S_2\cup\{0\}))=\thr^nS,\hspace{0.2cm}\forall n\in\mathbb{N}
		\end{equation}
		But $S_1\cup\{0\}$ and $S_2\cup\{0\}$ are each subsets of Euclidean quadrants. Moreover the metric $d(\cdot,\cdot)$ coincides with the usual Euclidean metric. By previous example $\deg_{\thr}(S_i\cup\{0\})\leq 2$ for $i=1,2$. By equation \eqref{eq:threq} it follows that $\deg_{\thr}S\leq 2$. Since $S\subseteq X$ is arbitrary then $\deg_{\thr}X\leq 2$. Therefore $(X,d)$ is of finite type.}
\end{example}

\section{Fr\'echet Mean}
\label{s:Frechet}

\subsection{A general convex optimization problem}
Let $(X,d)$ be a complete $\CAT(0)$ space and $x_1,x_2,...,x_n\in X$. For a given set of non-negative numbers $w_1,w_2,...,w_n\in[0,1]$ consider the optimization problem
\begin{equation}
\label{eq:cvxopt}
\arg\min_{x\in X} F_p(x)\hspace{0.2cm}\text{where}\hspace{0.2cm}F_p(x)\coloneqq\sum_{i=1}^nw_id(x,x_i)^p,\hspace{0.2cm}p\in[1,+\infty).
\end{equation}
The functional $F_p(x)$ is convex and continuous in $x$. When $p=1$ then $F_1$ becomes the objective function in the \textit{Fermat-Weber} problem for the optimal facility location. 
A minimizer of $F_1$ exists and it is known as the \textit{median} of the points $x_1,x_2,...,x_n$ with respect to the weights $w_1,w_2,...,w_n$. When $p=2$ then $F_2$ is the objective function in the \textit{Fr\'echet mean} problem. Because $d(x,x_i)^2$ is strongly convex and continuous in $x$ then $F_2$ has a unique minimizer which we denote by $x^*$. This unique minimizer is known as the \textit{Fr\'echet mean} of the points $x_1, x_2,...,x_n$ with respect to the weights $w_1, w_2,...,w_n$. 

\begin{lemma}
	\label{Frechet}
	Let $(X,d)$ be a complete $\CAT(0)$ space. The Fr\'echet mean of any finite set of points in $X$ lies in the closure of the convex hull of the given set.
\end{lemma}
\begin{proof}
	Let $S\coloneqq\{x_1,x_2,...,x_n\}$ be some finite set in $X$. Denote by $\cl\co S$ the closure of the convex hull of $S$. Suppose that the Fr\'echet mean $x^*\notin \co\cl S$. Because $\cl\co S$ is closed and convex by construction then the metric projection $P_{\cl\co S}x^*$ of $x^*$ onto $\cl\co S$ exists and it is unique. On the other hand we have the inequality 
	$$d(x^*,y)^2\geq d(x^*,P_{\cl\co S}x^*)^2+d(P_{\cl\co S}x^*,y)^2,\quad\forall y\in \cl\co S$$
	and in particular we must have
	$$d(x^*,x_i)^2\geq d(x^*,P_{\cl\co S}x^*)^2+d(P_{\cl\co S}x^*,x_i)^2,\quad\forall i=1,2,...,n.$$
	Therefore we obtain the inequality
	\begin{align*}
	\sum_{i=1}^nw_id(x^*,x_i)^2&\geq \sum_{i=1}^nw_id(x^*,P_{\cl\co S}x^*)^2+ \sum_{i=1}^nw_id(P_{\cl\co S}x^*,x_i)^2\\&>\sum_{i=1}^nw_id(P_{\cl\co S}x^*,x_i)^2
	\end{align*}
	which yields $F_2(P_{\cl\co S}x^*)<F_2(x^*)$. This is a contradiction. 
\end{proof}
\begin{corollary}
	Similarly the median of any finite set of points in $X$ lies in the closure of the convex hull of the given set.
\end{corollary}

\subsection{A constructability theorem}
In an Euclidean space $\mathbb{E}$ if $x_1,x_2,..., x_n\in\mathbb{E}$ and $x^*$ denotes their Fr\'echet mean then it is well known that 
\begin{equation}
\label{eq:fmeuclidean}
x^*=\frac{1}{n}\sum_{i=1}^nx_i
\end{equation}
In fact one gets to $x^*$ iteratively from any starting point in $x_{i_1}\in\{x_1,...,x_n\}$ as follows
\begin{align*}
& x'_1=\frac{1}{2}x_{i_1}+\frac{1}{2}x_{i_2},\quad x_{i_2}\in\{x_1,...,x_n\}\setminus\{x_{i_1}\}\\
& x'_2=\frac{2}{3}x'_1+\frac{1}{3}x_{i_3},\quad x_{i_3}\in\{x_1,...,x_n\}\setminus\{x_{i_1}, x_{i_2}\}\\
& ......\\
&x'_k=\frac{k}{k+1}x'_{k-1}+\frac{1}{k+1}x_{i_{k+1}},\quad x_{i_{k+1}}\in\{x_1,...,x_n\}\setminus\{x_{i_1},...,x_{i_{k}}\}\\
&......
\end{align*}
In particular when $k=n-1$ we obtain 
$$x'_{n-1}=\frac{n-1}{n}x'_{n-2}+\frac{1}{n}x_{i_n}$$
which by the means of the recursion is equivalent to formula 
$$x'_{n-1}=\frac{1}{n}\sum_{j=1}^nx_{i_j}$$
This in turn is equivalent to \eqref{eq:fmeuclidean}. One way to think about this recursion is in terms of threading. Notice that $x_1'$ lies in the geodesic segment $[x_{i_1},x_{i_2}]$ joining the elements $x_{i_1}$ and $x_{i_2}$ implying $x_1'\in\thr\{x_1,x_2,...,x_n\}$. And in general we have $x'_k\in[x'_{k-1},x_{i_{k+1}}]$ equivalently $x'_k\in\thr^k\{x_1,x_2,...,x_n\}$.
Given a finite set $S$ and $x\in X$ we say that $x$ is {\em constructible} from $S$ in a finite number of steps whenever $x\in\thr^nS$ for some $n\in\mathbb N$.
Motivated by above recursion and taking advantage of threading we state the following result for complete $\CAT(0)$ spaces of finite type.

\begin{thm}
	\label{complexity}
	Let $(X,d)$ be of finite type and $S\subseteq X$ a finite subset. Then the Fr\'echet mean $x^*$ of $S$ lies in $\co S$. Moreover $x^*$ is constructible from $S$ in at most $2^{\deg_{\thr}S}-1$ steps.
\end{thm}

\begin{proof}
	Let $S=\{x_1,x_2,...,x_n\}$ where $x_i\in X$ for all $i$. By Lemma \ref{Frechet} the Fr\'echet mean $x^*$ of $S$ lies in $\cl\co S$. Assumption $(X,d)$ is of finite type implies that $S$ has finite threading degree. If $\deg_{\thr}S=k$ for some $k\in\mathbb{N}$ then by definition this means that $\thr^kS=\co S$. But $S$ is finite and therefore compact. By Lemma \ref{thrcompact} it follows that $\co S$ is compact and therefore a closed set. Then $\co S=\cl\co S$ implies $x^*\in\co S$. Now $x\in \thr^kS$ means that there are $x^{k-1}_0,x^{k-1}_1\in\thr^{k-1}S$ such that $x\in[x^{k-1}_0,x^{k-1}_1]$. Then there are $x^{k-2}_0,x^{k-2}_1,x^{k-2}_2,x^{k-2}_3\in\thr^{k-2}S$ such that $x^{k-1}_0\in[x^{k-2}_0,x^{k-2}_1]$ and $x^{k-1}_1\in[x^{k-2}_2,x^{k-2}_3]$ and so on. In general for $0\leq m\leq k$ there exist $x^{k-m}_0,x^{k-m}_1,...,x^{k-m}_{2^m-1}\in\thr^{k-m}S$ such that $x^{k-m+1}_0\in[x_0^{k-m},x_{1}^{k-m}], x^{k-m+1}_1\in[x_2^{k-m},x_{3}^{k-m}],...,x^{k-m+1}_{2^{m-1}-1}\in[x_{2^m-2}^{k-m},x_{2^m-1}^{k-m}]$. Hence at each step $m$ we need to compute at most $2^{m-1}$ elements. This means that in total we have to construct at most
	$\sum_{m=1}^k2^{m-1}=2^{k}-1=2^{\deg_{\thr}S}-1$
	elements.
\end{proof}

%
%
%

%

\bibliographystyle{sn-mathphys}
\bibliography{literature}

\begin{thebibliography}{10}

\bibitem{Alexandrov}
A.~D. Alexandrov.
\newblock A theorem on triangles in a metric space and some of its
  applications.
\newblock {\em Trudy Mat. Inst. Steklova}, 38:5--23, 1951.

\bibitem{Border}
Ch.~D. Aliprantis and K.~C. Border.
\newblock {\em Infinite {D}imensional {A}nalysis}.
\newblock A {H}itchhiker's {G}uide. Springer-Verlag, 3 edition, 2006.

\bibitem{Ballman}
W.~Ballmann.
\newblock {\em Lectures on {S}paces of {N}onpositive {C}urvature}.
\newblock Birkh\"auser, 1995.

\bibitem{Banas}
J\'osef Bana\'s and Antonio Martin\'on.
\newblock Some properties of the {H}ausdorff distance in metric spaces.
\newblock {\em Bull. Aust. Math. Soc.}, 42:511--516, 1990.

\bibitem{Beer}
G.~Beer.
\newblock {\em Topologies on {C}losed and {C}losed {C}onvex {S}ets}, volume 268
  of {\em Mathematics and Its Applications}.
\newblock Kluwer Academic Publishers, 1993.

\bibitem{Berdellima}
A.~B\"erd\"ellima.
\newblock {\em Investigations in Hadamard spaces}.
\newblock PhD thesis, Georg-August-Universit\"at G\"ottingen, G\"ottingen,
  Germany, 2020.

\bibitem{BHV}
L.~J. Billera, S.~P. Holmes, and K.~Vogtmann.
\newblock Geometry of the space of phylogenetic trees.
\newblock {\em Adv. Appl. Math.}, 27:733--767, 2001.

\bibitem{Brid}
M.~R. Bridson and A.~Haefliger.
\newblock {\em Metric {S}paces of {N}onpositive {C}urvature}, volume 319 of
  {\em A Series of Comprehensive Studies in Mathematics}.
\newblock Birkh\"auser Boston Inc., Boston, 1999.

\bibitem{Bacak}
M.~Ba\' cak.
\newblock {\em Convex {A}nalysis and {O}ptimization in {H}adamard {S}paces},
  volume 22 of De Gruyter Series in Nonlinear Analysis and Applications.
\newblock De Gruyter, Berlin, 2014.

\bibitem{Gromov2}
M.~Gromov.
\newblock {\em Hyperbolic groups}, volume~8.
\newblock Essays in Group Theory. Math. Sci. Res. Inst. Publ., Springer, New
  York, NY, 1987.

\bibitem{Gromov}
M.~Gromov.
\newblock {\em Metric {S}tructures for {R}iemannian and non-{R}iemannian
  {S}paces}, volume 152.
\newblock Birkh\"auser Boston Inc., Boston, 1999 (Based on the 1981 French
  original).

\bibitem{Benoit}
B.~Kloeckner.
\newblock A geometric study of the {W}asserstein space of the line.
\newblock {\em ffhal-00275067v1}, 2008.

\bibitem{Reich}
E.~Kopeck\'a and S.~Reich.
\newblock Nonexpansive retracts in {Banach spaces}.
\newblock {\em J. Fixed Point Theory Appl.}, 77:161--174, 2007.

\bibitem{Kuratowski}
K.~Kuratowski.
\newblock {\em Topology}.
\newblock Academic Press, 1966.

\bibitem{OwenLub}
A.~Lubiw, D.~Maftuleac, and M.~Owen.
\newblock Shortest paths and convex hulls in 2{D} complexes with non-positive
  curvature.
\newblock {\em Comput. Geom.}, 89:51--91, 2020.

\bibitem{Lukenotes}
D.~R. Luke.
\newblock {\em Numerical {V}ariational {A}nalysis}.
\newblock Georg-August-Universit\"at G\"ottingen, 2019.

\bibitem{Jacob}
J.~Lurie.
\newblock {\em Notes on the Theory of {H}adamard Spaces}.
\newblock School of Mathematics, Institute for Advanced Study.

\bibitem{Rockafellar}
R.~T. Rockafellar.
\newblock {\em Convex {A}nalysis}.
\newblock Princeton Landmarks in Mathematics. Princeton University Press,
  Princeton, N.J., 1970.

\bibitem{Rockafellar2}
R.~T. Rockafellar and R.~J. Wets.
\newblock {\em Variational {A}nalysis}.
\newblock Grundlehren Math. Wiss. Springer-Verlag, Berlin, 1998.

\end{thebibliography}

%


\end{document}